\documentclass[a4paper,reqno]{amsart}
\usepackage{amscd,amsthm,amsfonts,amssymb,amsmath}
\usepackage{hyperref}
\usepackage[all]{xy}
\usepackage{mathrsfs}

\newcommand\lra{\longrightarrow}
\newcommand\inj{\hookrightarrow}
\newcommand\lto{\longmapsto}
\newcommand\ov{\overline}

\newcommand\hil[3]{\left(\frac{{#1},{#2}}{#3}\right)}
\newcommand\sqp{\sqrt[3]{p}}
\newcommand\leg[2]{\left(\frac{#1}{#2}\right)_3}
\newcommand\bfN{{\mathbf{N}}}
\newcommand\BQ{{\mathbb{Q}}}
\newcommand\BZ{{\mathbb{Z}}}
\newcommand\CO{{\mathcal{O}}}
\newcommand\fl{{\mathfrak{l}}}
\newcommand\fp{{\mathfrak{p}}}
\newcommand\Aut{{\mathrm{Aut}}}
\newcommand\Cl{{\mathrm{Cl}}}
\newcommand\Gal{{\mathrm{Gal}}}
\renewcommand\Im{{\mathrm{Im}}}
\newcommand\Ker{{\mathrm{Ker}\,}}
\newcommand\rk{{\mathrm{rk}}}
\newcommand\fm{{\mathfrak{m}}}

\newtheorem{thm}{Theorem}[section]
\newtheorem{theorem}[thm]{Theorem}
\newtheorem{lemma}[thm]{Lemma}
\newtheorem{proposition}[thm]{Proposition}
\numberwithin{equation}{section}

\title{The $3$-class groups of $\mathbb{Q}(\sqrt[3]{p})$ and its normal closure}

\author{Jianing Li}
\address{CAS Wu Wen-Tsun Key Laboratory of Mathematics, University of Science and Technology of China, Hefei, Anhui 230026, China}
\email{lijn@ustc.edu.cn}

\author{Shenxing Zhang}
\address{School of Mathematics, Hefei University of Technology, Hefei, Anhui 230009, China}
\email{zsxqq@mail.ustc.edu.cn}


\begin{document}
\maketitle

\begin{abstract}
We determine the $3$-class groups of $\BQ(\sqrt[3]{p})$ and $K=\BQ(\sqrt[3]{p},\sqrt{-3})$ when $p\equiv 4,7\bmod 9$ is a prime and $3$ is a cube modulo $p$.
This confirms a conjecture made by Barrucand-Cohn, and proves the last remaining case of a conjecture of Lemmermeyer on the $3$-class group of $K$.
\end{abstract}

\section{Introduction}

Let $p$ be a prime.
Let $F=\BQ(\sqp)$ and $K=\BQ(\sqp,\mu_3)$ the normal closure of $F$.
Let $A_F$ (resp. $A_K$) be the $3$-class group (i.e., $3$-Sylow subgroup of the class group) of $F$ (resp. $K$). The paper aims to prove the following result.

\begin{theorem}\label{thm:main}
Assume that $p\equiv 4,7\mod 9$ is a prime such that the cubic residue symbol $\leg{3}{p}=1$.
Then $A_F\cong \BZ/3\BZ$ and $A_K\cong(\BZ/3\BZ)^2$.
\end{theorem}

This result confirms a conjecture made by Barrucand-Cohn in \cite[\S 8]{BaC70}, and later mentioned by Barrucand-Williams-Baniuk, Williams and Gerth in \cite[\S 8, Conjecture 1]{BWB76}, \cite[p. 273]{Wil82} and \cite[p. 474]{Ger05}.
Theorem~\ref{thm:main} also completes a proof of a Lemmermeyer's conjecture on $A_K$ in \cite[Conjecture 5, \S 1.10]{Lem10} when combining with the following known results:

\begin{enumerate}
\item If $p\equiv 2\bmod 3$, then the groups $A_F$ and $A_K$ are both trivial; see \cite{Hon71}.
\item If $p\equiv 1\bmod 3$, then $A_F$ is cyclic non-trivial and $\rk(A_K) =1$ or $2$ where $\rk(A_K)$ is the $3$-rank of $A_K$; see \cite{Ger05}. 	
\item If $p\equiv 1\bmod 9$, then $\rk(A_K)=2$ if and only if $9$ divides $|A_F|$; see \cite[Lemma 5.11]{CaE05} and \cite{Ger05}.
\item If $p\equiv 4,7\bmod 9$ and $\leg{3}{p}\neq 1$; then $A_F \cong A_K\cong \BZ/3\BZ$; see \cite{BWB76} or \cite{Ger05}.
\end{enumerate}

We give two consequences of Theorem~\ref{thm:main}.
Let $E_K$ be the group of units of $K$.
Let $E'_K$ be the subgroup of $E_K$ generated by the units of non-trivial subfields of $K$.
Write $q=[E_K:E_{K'}]$. One has (\cite[Theorem 12.1, 14.1]{BaC71})
\[
q= 1 \text{ or } 3 \quad \text{ and } \quad h_K=\frac{q }{3}h^2_F.
\]
Here $h_K$ (resp. $h_F$) is the class number of $K$ (resp. $F$).
Thus, if $p\equiv 4,7\bmod 9$ and $\leg{3}{p}=1$, Theorem~\ref{thm:main} implies that $q=3$. This confirms a conjecture made in \cite{ATIA20}.

Assume $p\equiv 4,7\bmod 9$. Theorem~\ref{thm:main} implies that the norm equation $\bfN_{F/\BQ}(x)=3$ has a solution $x\in \CO_F$ if and only if $\leg{3}{p}=1$, as mentioned in \cite[p. 273]{Wil82}.
Since $\CO_F=\BZ[\sqrt[3]{p}]$, this is to say, the Diophantine equation 
	\[x^3_1+p x^3_2+p^2x^3_3-3p x_1x_2x_3=3\]
has solutions $(x_1,x_2,x_3)\in \BZ$ if and only if $\leg{3}{p}=1$.

\section{The proof}

\subsection{Chevalley's ambiguous class number formula}

We first review the $S$-version of Chevalley's ambiguous class number formula which will be used.
For a finite set $S$ of prime ideals of a number field $L$,  the {\em $S$-class group} of $K$ is defined as
	\[\Cl_{K,S}:=\Cl_K/\langle[\fp]:\fp\in S\rangle, \]
where $\Cl_K$ denotes the class group of $K$ and $[\fp]$ denotes the ideal class of $\fp$. Let $E_{K,S}:=\CO^\times_{K,S}$ denote the group of $S$-units of $K$. Let $L/K$ be a finite cyclic extension with Galois group $G$.
For a finite set $S$ of prime ideals of $K$, we denote by $\Cl_{L,S}=\Cl_{L,S_L}$ for simplicity, where $S_L$ is the set of primes of $L$ lying above those in $S$. Chevalley's ambiguous class number formula states that the order of the $G$-invariant subgroup of $\Cl_{L,S}$ is given by
	\begin{equation}\label{eq:amb}
		|\Cl_{L,S}^G|= |\Cl_{K,S}|\cdot\frac{\prod\limits_{v\notin S}e_v \cdot \prod\limits_{v\in S}e_v f_v }
			{[L:K]\cdot[E_{K,S}:E_{K,S}\cap \bfN L^\times]}.
	\end{equation}
Here the first product runs over all places of $K$ not in $S$, $e_v$ and $f_v$ are the ramification index and the residue degree of $v$ respectively, and $\bfN=\bfN_{L/K}$ is the norm map.
For a proof of this formula, see \cite{LiY20} for example. The unit index in \eqref{eq:amb} can be computed by Hilbert symbols provided that $L/K$ is a Kummer extension.

\begin{proposition}\label{prop:hasse}
Let $L/K$ be a cyclic extension of degree $d$ and $\mu_d\subset K$. Then $L=K(\sqrt[d]{a})$ for some $a\in K$. Let $\mathrm{Ram}$ be the set of ramified places of $K$. Define 
	\[\begin{split}
	\rho: \frac{E_{K,S}}{(E_{K,S})^d}&\lra\prod_{v \in S \cup \mathrm{Ram}} \mu_d\\
	x&\longmapsto\biggl(\hil{x}{a}{v}_d\biggr)_{v\in S \cup \mathrm{Ram} }.
	\end{split}\]
	Then the kernel of $\rho$ is given by
	\[\Ker \rho=\frac{E_{K,S}\cap \bfN L^\times}{(E_{K,S})^d}\]
	and hence the size of the image is given by
	\[|\Im(\rho)|=[E_{K,S}:E_{K,S}\cap \bfN L^\times],\]
	which is at most $d^{|S\cup \mathrm{Ram}|-1}$.
\end{proposition}

\begin{proof}
	This result is a standard direct consequence of local class field theory, Hasse's norm theorem, and the product formula for Hilbert symbols. For details, see \cite[\S 2]{LOXZ20}.
\end{proof}

If $\sigma\in \Aut(K)$ and $v$ is a prime of $K$, we have (\textit{loc. cit.})
\begin{equation}\label{eq:hilbert}
\sigma \hil{a}{b}{v}_d= \hil{\sigma(a)}{\sigma(b)}{\sigma(v)}_d,\quad a,b\in K^\times.
\end{equation}

For our applications, the degree $[L:K]$ is a power of a prime $\ell$.
For any finite generated abelian group $A$, we denote by $A_\ell=A\otimes\BZ_\ell$ where $\BZ_\ell$ is the ring of $\ell$-adic integers.
If $A$ is finite, $A_\ell$ is the $\ell$-primary subgroup of $A$.
If there is no ambiguity, we write $a$ for $a\otimes 1\in A_\ell$ for $a\in A$.
Clearly, the formula \eqref{eq:amb} still holds by replacing $(\Cl_{L,S})^G$ and $\Cl_{K,S}$ with $\bigl((\Cl_{L,S})_\ell\bigr)^G=(\Cl^G_{L,S})_\ell$ and $(\Cl_{K,S})_\ell$ respectively.

The following well known fact which is proved by counting the orbits of the $G$-action or by Nakayama's Lemma will be used frequently:
	\[  (\Cl_{L,S})_\ell=0 \text{ if and only if } (\Cl^G_{L,S})_\ell=0.\]

\subsection{Proof of Theorem~\ref{thm:main}}

From now on, assume $p\equiv 1\bmod 3$.
Denote by
\begin{itemize}
	\item $k=\BQ(\mu_3)$;
	\item $M^+$ the unique cubic subfield of $\BQ(\mu_p)$, which is real;
	\item $M=M(\mu_3)$ a quadratic extension of $M^+$;
	\item $L=KM=M(\sqp,\mu_3)$;
	\item $A_T=(\Cl_T)_3$ for any number field $T$.
\end{itemize}

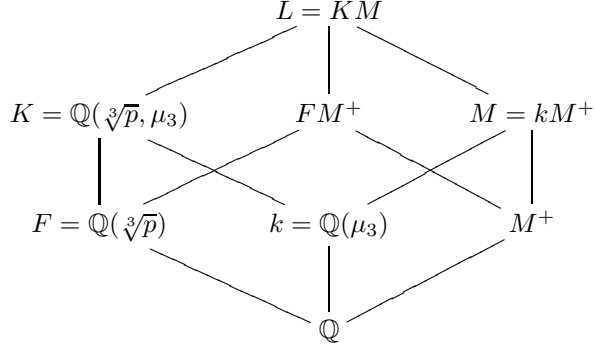
	\[\xymatrix{
																		&L=KM																&\\
		K=\BQ(\sqp,\mu_3)\ar@{-}[ru]		&FM^+\ar@{-}[u]											&M=kM^+\ar@{-}[lu]\\
		F=\BQ(\sqp)\ar@{-}[u]\ar@{-}[ru]&k=\BQ(\mu_3)\ar@{-}[lu]\ar@{-}[ru]	&M^+\ar@{-}[u]\ar@{-}[lu]\\
																		&\BQ\ar@{-}[u]\ar@{-}[lu]\ar@{-}[ru]&
	}\]

\begin{proposition}\label{prop:cl_m}
Assume $p\equiv 1\bmod 3$.

(1) There exists $\alpha\in\CO_k$ such that $M=k(\sqrt[3]{p\alpha})$ and $p=\alpha\ov\alpha$;

(2) $A_M=0$ if and only if $p\equiv 4,7\bmod 9$.
\end{proposition}
\begin{proof}
(1)	Since $p\equiv 1\bmod 3$, we can write $p=\alpha\ov \alpha$ for some $\alpha \in \CO_k$.
Note that $(\alpha)$ and $(\ov \alpha)$ are exactly the ramified primes of $k$ in $M$.
Now, since the class number of $k$ is $1$ and $M/k$ is a Kummer extension, we have
	\[M=\BQ(\sqrt[3]{\gamma}) \quad \text{and} \quad \gamma=\zeta_3^a \alpha^{b} \ov\alpha^c\]
with $a,b,c \in \{0,1,2\}$ and $bc\neq 0$.
Since $(3-a,3-b,3-c)$ gives the same field as $(a,b,c)$, 
we conclude that $M=k(\sqrt[3]{\zeta^a_3p})$ or $k(\sqrt[3]{\zeta^a_3p\alpha})$ for some $a=0,1,2$.
Since $M$ is abelian over $\BQ$ but $k(\sqrt[3]{\zeta^a_3p})/\BQ$ is not, $M$ must coincide with $k(\sqrt[3]{\zeta^a_3p\alpha})$.
By replacing $\alpha$ with $\zeta^a_3\alpha$, we have $M=k(\sqrt[3]{p\alpha})$ and $p=\alpha\bar{\alpha}$.
This proves (1).

(2) We apply \eqref{eq:amb} and Proposition~\ref{prop:hasse} to the cyclic cubic extension $M/k$. 
Let $\iota:k\inj\BQ_p$ be the embedding induced by $(\alpha)$.
Then we have the following equalities of cubic Hilbert symbols:
	\begin{equation}\label{eq:hil1}
		\hil{\zeta_3}{p\alpha}{\alpha}=\iota^{-1}\hil{\iota(\zeta_3)}{p\iota(\alpha)}{\BQ_p}
				=\iota^{-1}\hil{\iota(\zeta_3)}{\iota(\alpha)}{\BQ_p}^{-1}=\zeta_3^{(p-1)/3}.
	\end{equation}
Hence this symbol as well as the index $[E_k: E_k\cap \bfN M^\times]$ is trivial if and only if $p\equiv 1\bmod 9$.
Thus
	\[|A_M^G|=\frac{3^2}{3\cdot[E_k:E_k\cap \bfN M^\times]}=1\]
if and only if $p\equiv 4,7\bmod 9$.
By Nakayama's lemma, it turns out that $A_M$ is trivial if and only if $p\equiv 4,7\bmod 9$. This completes the proof of Proposition~\ref{prop:cl_m}. 
\end{proof}

Let $\fp$ (resp. $\fp'$) be the unique prime of $M$ (resp. $K$) lying above $\alpha\CO_k$.
Then $\alpha \CO_M = \fp^3$ and $\alpha\CO_K={\fp'}^3$.

\begin{proposition}\label{prop:split}
Assume that $p\equiv 1\bmod 3$ and $\leg{3}{p}=1$.

(1) The extensions $L/K$ and $FM^+/F$ are both abelian unramified cubic extensions.

(2) The primes $\fp$ and $\fp'$ both split in $L$.
\end{proposition}

\begin{proof}
(1) Since $L=KM^{+}$, we have that $L/K$ is unramified outside the primes above $p$.
Denote by $I_{(\alpha)}$ the inertia group of $(\alpha)=\alpha\CO_k$ in the abelian extension $L/k$.
By local class field theory and noting that the completion of $k$ at $(\alpha)$ is $\BQ_p$, we have a surjection 
	\[\BZ^\times_p \twoheadrightarrow I_{(\alpha)}.\]
It follows that $I_{(\alpha)}$ can not be $\Gal(L/k)\cong(\BZ/3\BZ)^2$.
On the other hand, $I_{(\alpha)}$ is non-trivial since $(\alpha)$ is ramified in $K$ and $M$.
This shows that $\fp$ and $\fp'$ must be unramified in $L$.
An entirely similar argument for the prime $(\ov\alpha)=\ov\alpha\CO_k$ shows that $L/M$ and $L/K$ are both unramified outside the primes above $\ov\alpha$. 
This shows that $L/K$ is unramified everywhere. 

For the extension $FM^+/F$, first note that it is unramified outside $\sqrt[3]{p}\CO_F$ as $M^+/\BQ$ is unramified outside $p$. 
We claim that $\sqrt[3]{p}\CO_F$ is also unramified in $FM^+$. 
Otherwise, since $K/F$ is unramified at $\sqrt[3]{p}\CO_F$, the prime of $K$ above $\sqrt[3]{p}$ would be ramified in $L$. 
But this contradicts that $L/K$ is unramified whence the claim holds. 
This proves (1).

(2) We have just shown that $FM^+$ is contained in the Hilbert class field of $F$.
By class field theory, the principal prime $\sqrt[3]{p}\CO_F$ splits in $FM^{+}$.
It follows that $\fp'$ and $\fp$ both split in $L$. 
\end{proof}

\begin{lemma}\label{lem:ram}
(1) If $p\equiv 4,7\bmod 9$, then $3$ is totally ramified in $K$.

(2) If $p\equiv 1\bmod 3$ and $3$ is a cube modulo $p$, then $(1-\zeta_3)\CO_k$ splits in $M$.
\end{lemma}
\begin{proof}
(1) Since $(x+p)^3-p$ is an Eisenstein polynomial, $3$ is totally ramified in $F$.
Since $3$ is also ramified in $k$, it follows that $3$ is totally ramified in $K$ by counting the ramification degrees.

(2) Fix the canonical isomorphism 
	\[\begin{split}
		(\BZ/p\BZ)^\times &\cong \Gal(\BQ(\mu_p)/\BQ)\\
		a&\mapsto(\sigma_a:\zeta_p \mapsto \zeta^a_p).
	\end{split}\]
By definition, $M^+$ is the subfield of $\BQ(\mu_p)$ fixed by $(\BZ/p\BZ)^{\times3}$. 
Our assumptions imply that $\sigma_3$ is trivial on $M^+$ whence $3$ splits in $M^{+}$. It follows that $(1-\zeta_3)\CO_k$ must split in $M$.
\end{proof}

We need the following elementary fact on the local field $\BQ_3(\mu_3)$.
\begin{lemma}\label{lem:hil_local}
If $a,b \in \BZ$ with $3\nmid ab$, then the cubic Hilbert symbol of $a$ and $b$ in $\BQ_3(\mu_3)$ is trivial.
\end{lemma}
\begin{proof}
By convergence of the Taylor expansion of $(1+9x)^{1/3}$ on $\BZ_3[\mu_3]$, every element in $1+9\BZ_3[\mu_3]$ is a cube.
Note that $-1$ is a cube whence the cubic Hilbert symbol $\hil{a}{a}{\BQ_3(\mu_3)}=1$.
Thus, we only need to show the triviality of the symbol
	\[\hil{4}{7}{\BQ_3(\zeta_3)}=\hil{4}{2}{\BQ_3(\zeta_3)}=\hil{2}{2}{\BQ_3(\zeta_3)}^2=1.\qedhere\]
\end{proof}

\begin{theorem}\label{thm:main2}
Assume that $p\equiv 4,7\bmod 9$ and $\leg{3}{p}=1$.
Then $A_L$ is non-trivial and $3$ does not divides $|\Cl_{L,\{\fp \}}|$.
\end{theorem}

\begin{proof}
We first apply \eqref{eq:amb} on $L/M$ with $S=\emptyset$ to prove $3$ divides $|A^G_L|$ where $G=\Gal(L/M)$. 
By Proposition~\ref{prop:split} and Lemma~\ref{lem:ram}, exactly the three primes $\fl,\sigma(\fl),\sigma^2(\fl)$ of $M$ lying above $(1-\zeta_3)\CO_k$ are ramified in $L/M$, where $\sigma$ is a generator of $\Gal(M/k)$.
By Proposition~\ref{prop:cl_m}, we know that $|A_M|=1$.
It remains to compute the unit index.
Note that $L=M(\sqrt[3]{p})$. 
To apply Proposition~\ref{prop:hasse}, we define
	\[\begin{split}
		\rho:E_M&\lra\mu_3^3\\
		u&\lto\biggl(\hil{u}{p}{\fl}, \hil{u}{p}{\sigma(\fl)}, \hil{u}{p}{\sigma^2(\fl)}\biggr).
	\end{split}\]
Since $M/M^+$ is a CM-extension, the group $(E_{M})_3$ is generated by $(E_{M^+})_3$ and $\zeta_3$ by \cite[Theorem 4.12]{Was82}. The completion of $M^{+}$ at a prime above $3$ is $\BQ_3$. It follows that $\eta \equiv a \bmod 9$ with $a\in \BZ$ for any $\eta \in E_{M^+}$. Then by Lemma~\ref{lem:hil_local}, 
	\begin{equation}\label{eq:trivial}
		|\rho( E_{M^+}) |=1.
	\end{equation}
Now we compute $\rho(\zeta_3)$. 
Since $\sigma(\zeta_3)=\zeta_3$, by \eqref{eq:hilbert} we have 
	\[\hil{\zeta_3}{p}{\fl}=\hil{\zeta_3}{p}{\sigma(\fl)}=\hil{\zeta_3}{p}{\sigma^2(\fl)}.\]
By Lemma~\ref{lem:ram}, the completion of $M$ at $\fl$ is $\BQ_3(\mu_3)$. Applying the product formula for cubic Hilbert symbols on the field $\BQ(\mu_3)$ gives
	\[\hil{\zeta_3}{p}{(\alpha)} \hil{\zeta_3}{p}{(\ov\alpha)} \hil{\zeta_3}{p}{(1-\zeta_3)}=1. \]
By \eqref{eq:hil1} and our assumption $p\equiv 4,7\bmod 9$, we obtain that
 \begin{equation}\label{eq:non-trivial}
	 \hil{\zeta_3}{p}{(1-\zeta_3)}\neq 1 \text{ and } \hil{\zeta_3}{p}{\BQ_3(\mu_3)}\neq 1.
 \end{equation}
This proves that $\rho(\zeta_3) =\zeta^{\pm 1}_3(1,1,1)$. 
Combining with \eqref{eq:trivial}, we conclude that $|\rho(E_M)_3|=3$.
Then Chevalley's formula gives
	\[|A_L^G|=\frac{3^3}{3\times 3}=3.\]
In particular, $|A_L|\ge 3$. 

Next, we apply Chevalley's formula on $L/M$ with $S=\{\fp \}$ to compute $\Cl^G_{L,\{\fp \}}$.
Define
	\[\beta=\frac{\sqrt[3]{p\alpha}}{\bfN_{\BQ(\mu_p)/M^+}(1-\zeta_p)}.\]
Note that $\beta^3$ generates the ideal $\alpha\CO_M$ whence $\beta \CO_M=\fp$.
It follows that $(E_{M,\{\fp \}})_3$ is generated by $\beta, \zeta_3$ and $E_{M^+}$.
We claim that 
	\[\hil{\beta}{p}{\fl} \neq  \hil{\beta}{p}{\sigma(\fl)}.\]
Indeed, by \eqref{eq:hilbert}, the right hand side equals the Hilbert symbol of $\sigma^{-1} (\beta)$ and $p$ at $\fl$.
Note that $\sigma^{-1}(\beta) = \zeta^{\pm 1}_3 \beta \eta$ for some $\eta\in E_{M^+}$.
Thus the inequality follows from \eqref{eq:trivial} and \eqref{eq:non-trivial}.
By Proposition~\ref{prop:hasse}, this shows that the index 
	\[[E_{M, \{\fp \}} : E_{M,\{ \fp \} }\cap \bfN L^\times ] = 9.\]
By Proposition~\ref{prop:split}, the prime $\fp$ splits in $L$.
It follows from \eqref{eq:amb} that $3$ does not divide $|\Cl^G_{L, \{\fp \}}|$ whence $3$ does not divide $|\Cl_{L, \{\fp \}}|$ by Nakayama's Lemma.
This completes the proof. 
\end{proof}

\begin{proof}[Proof of Theorem~\ref{thm:main}]
By Theorem~\ref{thm:main2}, $A_L$ is non-trivial. It follows that, by Nakayama's lemma, we have $|A^{\Gal(L/K)}_L|\ge 3$.
Since $L/K$ is unramified everywhere, by Hasse's norm theorem and local class field theory (or Proposition~\ref{prop:split}),  we have the unit index $[E_K:E_K\cap \bfN(L^\times)]=1$. Then applying Chevalley's formula with $S=\emptyset$ to the extension $L/K$ gives
	\[|A_K|\ge 9.\]

Recall that $\fp'$ is the prime of $K$ lying above $\alpha\CO_k$.
Note that $\fp' \CO_L=\fp \CO_L$, we have $\Cl_{L, \{\fp' \}}= \Cl_{L, \{\fp \}}$ by definition.
Since $3$ does not divide $|\Cl_{L, \{\fp' \}}|$ by Theorem~\ref{thm:main2} and $\fp'$ splits in $L$ by Proposition~\ref{prop:split}, Chevalley's formula with $S=\{\fp' \}$ will imply that $(\Cl_{K,\{ \fp' \}})_3 \cong \BZ/3\BZ$ if we can show that 
	\[[E_{K, \{\fp' \} }: E_{K, \{\fp' \}}\cap \bfN(L^\times)]=1.\]
Because $\fp'$ splits in $L/K$ by Proposition~\ref{prop:split}, the local extension at $\fp'$ is trivial.
Thus any $\{\fp' \}$-unit is a local norm at $\fp'$ whence is a local norm at every place of $K$ as $L/K$ is unramified.
By Hasse's norm theorem, the above unit index is indeed trivial.

The equality $\fp'^3 = \alpha \CO_K$ implies that $|\Cl_{K}|\leq 3|\Cl_{K,\{ \fp' \}}|$. 
It follows that
	\[|A_K|\leq 9.\]
Hence $|A_K|=9$ and then $|A_F|=3$ by \cite[Lemma~1]{Hon71}.

Let $\tau$ be the non-trivial element of $\Delta=\Gal(K/F)$.
Since $\Delta$ is of order $2$, we have a decomposition of $\BZ_3[\Delta]$-modules
	\[A_K=A_K^+\oplus A_K^-, \text{ where } A^{\pm}_K = \{ a\in A_K| \tau(a) = a^{\pm 1} \}.\]
It is well known  that $|A_K^+|= |A_F|=3$ (for example, using \eqref{eq:amb}).
Thus $A_K$ has a direct factor $\BZ/3\BZ$.
This implies that $A_K\cong(\BZ/3\BZ)^2$, completing the proof of Theorem~\ref{thm:main}.
\end{proof}

\medskip

Ren\'{e} Schoof informed us that, using Theorem~\ref{thm:main}, one can go back to improve Theorem~\ref{thm:main2}. Under the assumptions of Theorem~\ref{thm:main2}, we indeed have 
\begin{equation}\label{eq:schoof1}
A_L \cong \BZ/3\BZ. 
\end{equation}
We write his argument as follows. In what follows, let $G=\Gal(L/k)$ so that $G\cong (\BZ/3\BZ)^2$. 
\begin{lemma}\label{lem:schoof1}
For any subgroup $H$ of $G$ of order $3$, we have $|A^{H}_L| =3$.
\end{lemma}

\begin{proof}
The nontrivial intermediate fields of $L/k$ are $K,M, k(\sqrt[3]{\alpha}), k(\sqrt[3]{\bar{\alpha}})$, where $\alpha$ is as in Proposition~\ref{prop:cl_m}. The case $H=\Gal(L/M)$ has been proved in the proof of Theorem~\ref{thm:main2}.
For $H=\Gal(L/K)$, as in proof of Theorem~\ref{thm:main}, one has $|A^H_L|=\frac{|A_K|}{3}=3$; here the last equality is by Theorem~\ref{thm:main}.

Now consider the case $H=\Gal(L/k(\sqrt[3]{\alpha}))$. We first prove $A_{k(\sqrt[3]{\alpha})}=0$ by the same method as in Proposition~\ref{prop:cl_m}. There are precisely two primes $(\alpha), (1-\zeta_3)$ of $k$ ramified in $k(\sqrt[3]{\alpha})/k$. Since $p\equiv 4,7\bmod 9$, we have
$\hil{\zeta_3}{\alpha}{(\alpha)}\neq 1$
as in \eqref{eq:hil1}.
Then Chevalley's formula for $k(\sqrt[3]{\alpha})/k$ and Nakayama's lemma give $A_{k(\sqrt[3]\alpha)}=0$. Now there are precisely three primes of $k(\sqrt[3]{\alpha})$ ramified in $L/k(\sqrt[3]{\alpha})$ which are the primes lying above $\bar{\alpha}$. In particular, the completion at such a prime $\mathfrak{P}$ is isomorphic to $\BQ_p$. Note that $L=k(\sqrt[3]{\alpha}) (\sqrt[3]{p})$. Hence  $\hil{\zeta_3}{p}{\mathfrak{P}}\neq 1$ as $p\equiv 4,7\bmod 9$. Applying Chevalley's formula to $L/k(\sqrt[3]{\alpha})$ gives $|A^H_L|\leq 3$ whence it is $3$ by Nakayama's lemma. The case $H=\Gal(L/k(\sqrt[3]{\alpha}))$ can be proved in an entirely similar way. \end{proof}

\begin{lemma}[Schoof]\label{lem:schoof2}
Let $R$ be a complete local Noetherian ring with maximal ideal $\fm$. Suppose that $\dim \fm/\fm^2  > 1$.  Let $J \subset \fm$  be an ideal for which $J + (x) = \fm$  for every $x\in \fm - \fm^2$. Then $J =\fm$.
\end{lemma}

\begin{proof}
We may replace $J$ by $J + \fm^2$, because if $J+\fm^2$ is equal to $\fm$, then also  $J = \fm$  by Nakayama's lemma. So we have  $\fm^2 \subset J$. Then $W = J/\fm^2$ is a sub vector space of $V = \fm/\fm^2$ and we know  that $W + v = V$ for every non-zero   vector $v  \in V$. Since $\dim V > 1$, $W$ cannot be zero. So we can find a non-zero $v \in W$. It follows that $W = W + v =V$ and we are done.
\end{proof}

\begin{proof}[Proof of \eqref{eq:schoof1}]
Clearly, $A_L$ is a module over the complete local ring  $R = \BZ_3[G]$. The maximal ideal $\fm$ of $R$ is generated by  the elements  $s-1$ with   $s \in G$. By Theorem~\ref{thm:main2}, $A_L$ is cyclic over the ring $R$. So,  $A_L = R/J$ for some ideal $J\subset R$. Let $s\in G$ be a nontrivial element and write $H$ for the subgroup generated by $s$.
We have an exact sequence
\[ 0\to A^{H}_L \to A_L \xrightarrow[]{s-1} A_L \to A_L/(s-1)A_L \to 0.\]
The rightmost term  has order $3$ for every nontrivial $s\in G$ by Lemma~\ref{lem:schoof1}. In other words, $J+(s-1)=\fm$ for every nontrivial $s\in G$. Since $\fm/\fm^2$ is generated by the elements $s-1$ ($s\in G$) as a $R/\fm$-vector space, it follows from Lemma~\ref{lem:schoof2} that $J=\fm$. Hence $|A_L|=3$.
\end{proof}

%

\textbf{Acknowledgments.}
The authors are very grateful to Ren\'{e} Schoof for informing us the above improvement of Theorem~\ref{thm:main2}.
The authors also would like to thank Franz Lemmermeyer and Yi Ouyang for helpful comments on this article.
The authors are partially supported by Anhui Initiative in Quantum Information Technologies (Grant No. AHY150200) and by the Fundamental Research Funds for the Central Universities (Grant No. WK3470000020 and Grant No. WK0010000061 respectively).
The second named author is also supported by NSFC (Grant No. 12001510).

\end{document}